\documentclass[a4paper,reqno]{amsart}                

\usepackage[latin1]{inputenc}                  
\usepackage[T1]{fontenc}                       
\usepackage[spanish,english]{babel}            
\usepackage{amsmath,amssymb,amsthm}            
\usepackage{latexsym}                          
\usepackage{delarray}                          
\usepackage{bbm}                               
\usepackage{hyperref}                          
\usepackage[pdftex,usenames,dvipsnames]{color} 
\usepackage{bbding,trfsigns}                   
\usepackage{wasysym}                           
\usepackage{datetime}                          
\usepackage{cases}

\DeclareMathAlphabet{\mathpzc}{OT1}{pzc}{m}{it} 

\newtheorem{Th}{Theorem}[section]              
\newtheorem{Cor}[Th]{Corollary}

\newtheorem{Prop}[Th]{Proposition}
\newtheorem{Lem}[Th]{Lemma}

\newcommand{\G}{\Gamma}
\newcommand{\LLL}{\mathbb{L}}

\newcommand{\W}{\mathbb{W}}

\DeclareMathOperator{\spann}{span}

\title[$L^p$-boundedness for Bessel and Laguerre maximal operators]
      {$L^p$-boundedness properties for the maximal operators for semigroups associated with Bessel and Laguerre operators}

\author[J.J. Betancor]{J.J. Betancor}
\author[A.J. Castro]{A.J. Castro}
\author[P.L. De Nápoli]{P.L. De Nápoli}
\author[J.C. Fariña]{J.C. Fariña}
\author[L. Rodríguez-Mesa]{L. Rodríguez-Mesa}

\address{\newline
        Jorge J. Betancor, Alejandro J. Castro, Juan C. Fariña and Lourdes Rodríguez-Mesa \newline
        Departamento de Análisis Matemático,
        Universidad de la Laguna, \newline
        Campus de Anchieta, Avda. Astrofísico Francisco Sánchez, s/n, \newline
        38271, La Laguna (Sta. Cruz de Tenerife), Spain}
\email{jbetanco@ull.es, ajcastro@ull.es, jcfarina@ull.es, lrguez@ull.es}

\address{\newline
        Pablo L. De Nápoli \newline
        Departamento de Matem\'atica \newline
        Universidad de Buenos Aires, \newline
        e Instituto de Investigaciones Matem\'aticas \newline
        ``Luis A. Santal\'o'', CONICET \newline
        1248 Pabell\'on 1, Ciudad Universitaria, Buenos Aires, Argentina}
\email{pdenapo@dm.uba.ar}



\thanks{The authors are partially supported by MTM2010/17974.
        The second author is also supported by a FPU grant from the Government of Spain.
        The third author is also partially supported by Conicet (Argentina) under PIP
        1420090100230, by ANPCYT under PICT 01307 and by UBACyT research proyect: 20020090 100067.}

\begin{document}


  \maketitle                

  \begin{abstract}
    In this paper we prove that the generalized (in the sense of Caffarelli and Calderón \cite{CC1})
    maximal operators associated with heat semigroups for Bessel and Laguerre operators are weak type
    $(1,1)$. Our results include other known ones and our proofs are simpler than the ones for the
    known special cases.
  \end{abstract}

    \section{Introduction} \label{sec:intro}

    Stein investigated in \cite{Ste1} harmonic analysis associated to diffusion semigroups of operators.
    If $\{T_t\}_{t>0}$ is a diffusion semigroup in the measure space $(\Omega,\mu)$, in
    \cite[p. 73]{Ste1} it was proved that the maximal operator $T_*$ defined by
    $$T_* f = \sup_{t>0} |T_t f|$$
    is bounded from $L^p(\Omega,\mu)$ into itself, for every $1<p<\infty$. As far as we know there is not a
    result showing the behavior of $T_*$ on $L^1(\Omega,\mu)$ for every diffusion semigroup $\{T_t\}_{t>0}$.
    The behavior of $T_*$ on $L^1(\Omega,\mu)$ must be established by
    taking into account the intrinsic properties of $\{T_t\}_{t>0}$. The usual result says that
    $T_*$ is bounded from $L^1(\Omega,\mu)$ into $L^{1,\infty}(\Omega,\mu)$, but not bounded from
    $L^1(\Omega,\mu)$ into $L^1(\Omega,\mu)$. In order to analyze $T_*$ in $L^1(\Omega,\mu)$ in many cases
    this maximal operator is controlled by a Hardy-Littlewood type maximal operator, and also, the vector
    valued Calderón-Zygmund theory (\cite{RRT}) can be used. These procedures have been employed to study
    the maximal operators associated to the classical heat semigroup \cite[p. 57]{Ste2},
    to Hermite operators (\cite{Mu}, \cite{Sj} and \cite{ST1}), to Laguerre operators
    (\cite{MST1}, \cite{MST2}, \cite{Mu}, \cite{NS1} and \cite{Stem2}), to Bessel operators
    (\cite{BCC}, \cite{BCN}, \cite{BHNV}, \cite{MS} and \cite{Stem0}) and to Jacobi operators (\cite{MS} and \cite{NS2}),
    amongst others.\\

    Our objective in this paper is to study the $L^p$-boundedness properties, $1 \leq p \leq \infty$, for the
    generalized (in the sense of Caffarelli and Calderón \cite{CC1}) maximal operators associated to the multidimensional
    Bessel and Laguerre operators.\\

    Our results (see Theorems below) extend the others known for the Bessel operators
    (\cite[Theorem 2.1]{BHNV}
    and \cite[Theorem 1.1]{BCC})
    and for the Laguerre operators (\cite[Theorem 1.1]{NS1}). Moreover, by exploiting  ideas developed by
    Caffarelli and Calderón (\cite{CC1} and \cite{CC2}) we are able to prove our result in a much simpler
    way than the one followed in \cite{BCC}, \cite{BHNV} and \cite{NS1}.\\

    We now recall some definitions and properties in the Bessel and Laguerre settings which allow us to
    state our results.\\

    We consider for $\lambda>-1/2$, the Bessel operator $\Delta_\lambda$  defined by
    $$\Delta_\lambda
        = -x^{-2\lambda} \frac{d}{dx} x^{2\lambda} \frac{d}{dx}
        = - \frac{d^2}{dx^2}-\frac{2\lambda}{x} \frac{d}{dx},
        \quad \text{on } (0,\infty),$$
    and, if $J_\nu$ represents the Bessel function of the first kind and order $\nu$, the Hankel
    transformation $h_\lambda$ is given by
    $$h_\lambda(f)(x)
        = \int_0^\infty (xy)^{-\lambda+1/2} J_{\lambda-1/2}(xy)f(y)y^{2\lambda}dy,
        \quad x \in (0,\infty),$$
    for every $f \in L^1((0,\infty),x^{2\lambda}dx)$. $h_\lambda$ can be extended to
    $L^2((0,\infty),x^{2\lambda}dx)$ as an isometry in $L^2((0,\infty),x^{2\lambda}dx)$
    and $h_\lambda^{-1}=h_\lambda$. If $f \in C_c^\infty(0,\infty)$ we have that
    $$h_\lambda(\Delta_\lambda f)(x)
        = x^2 h_\lambda(f)(x),
        \quad x \in (0,\infty).$$
    This property suggests to extend the definition of $\Delta_\lambda$ as follows
    $$\Delta_\lambda f
        = h_\lambda (x^2 h_\lambda (f)),
        \quad f \in D(\Delta_\lambda),$$
    where
    $$D(\Delta_\lambda)
        = \{f \in L^2((0,\infty),x^{2\lambda}dx) : x^2 h_\lambda(f) \in L^2((0,\infty),x^{2\lambda}dx)\}.$$
    Thus, $\Delta_\lambda$ is a positive and selfadjoint operator. Moreover, $-\Delta_\lambda$
    generates a semigroup of operators $\{W_t^\lambda\}_{t>0}$ in $L^2((0,\infty),x^{2\lambda}dx)$
    where
    \begin{equation}\label{0.1}
        W_t^\lambda(f)
            = h_\lambda \left( e^{-ty^2} h_\lambda(f)\right),
            \quad f \in L^2((0,\infty),x^{2\lambda}dx) \text{ and } t>0.
    \end{equation}
    According to \cite[p. 195]{Wat} we can write, for $f \in L^2((0,\infty),x^{2\lambda}dx)$
    \begin{equation}\label{1.1}
        W_t^\lambda(f)(x)
            = \int_0^\infty W_t^\lambda(x,y) f(y) y^{2\lambda} dy,
            \quad   x,t \in (0,\infty),
    \end{equation}
    where the Hankel heat kernel semigroup $W_t^\lambda(x,y)$ is defined by
    $$W_t^\lambda(x,y)
        = \frac{(xy)^{-\lambda+1/2}}{2t} I_{\lambda-1/2}\left( \frac{xy}{2t} \right) e^{-(x^2+y^2)/4t},
        \quad x,y,t \in (0,\infty),$$
    and $I_\nu$ denotes the modified Bessel function of the first kind and order $\nu$.\\

    Since $\int_0^\infty W_t^\lambda(x,y) y^{2\lambda} dy=1$,
    $x,t \in (0,\infty)$, $\{W_t^\lambda\}_{t>0}$ defined by \eqref{1.1} is a diffusion semigroup in
    $L^p((0,\infty),x^{2\lambda}dx)$, $1 \leq p \leq \infty$.\\

    Suppose now that $\lambda=(\lambda_1, \dots, \lambda_n) \in (-1/2,\infty)^n$. We define the
    $n$-dimensional Bessel operator $\Delta_\lambda$ by
    $$\Delta_\lambda = \sum_{j=1}^n \Delta_{\lambda_j, x_j}.$$
    The operator $-\Delta_\lambda$ generates the diffusion semigroup $\{\W_t^\lambda\}_{t>0}$
    in $L^p((0,\infty)^n,d\mu_\lambda)$, $1 \leq p \leq \infty$, where
    $d\mu_\lambda(x) = \overset{n}{\underset{j=1}{\prod}} x_j^{2 \lambda_j} dx_j $,
    $x=(x_1, \dots, x_n) \in (0,\infty)^n$ and
    $$\W_t^\lambda(f)(x)
        = \int_{(0,\infty)^n} \W_t^\lambda(x,y) f(y) d\mu_\lambda(y),
          \quad f \in L^p((0,\infty)^n, d\mu_\lambda)  \text{ and } x,t \in (0,\infty), $$
    being
    $$\W_t^\lambda(x,y)
        = \prod_{j=1}^n W_t^{\lambda_j}(x_j,y_j),
        \quad x, y \in (0,\infty)^n \text{ and } t>0.$$
    The maximal operator $\W_*^\lambda$ associated with $\{\W_t^\lambda\}_{t>0}$ is defined by
    $$\W_*^\lambda(f)
       = \sup_{t>0} |\W_t^\lambda(f)|. $$
    In \cite[Theorem 1.1]{BCC} (also in \cite[Theorem 2.1]{BCN} when $\lambda \in (0,\infty)^n$
    and in \cite[Theorem 2.1]{BHNV} for $n=1$) it was proved that
    $\W_*^\lambda$ is a bounded operator from $L^1((0,\infty)^n, d\mu_\lambda)$ into
    $L^{1,\infty}((0,\infty)^n, d\mu_\lambda)$. Note that, since $\{\W_t^\lambda\}_{t>0}$ is a diffusion
    semigroup $\W_*^\lambda$ is bounded from $L^p((0,\infty)^n, d\mu_\lambda)$ into itself,
    for every $1<p \leq \infty$ (see \cite[p. 73]{Ste1}).\\

    Motivated by \cite{CC1} we consider a function $r=(r_1, \dots, r_n)$ where, for every
    $j=1, \dots, n$, $r_j : [0,\infty) \longrightarrow [0,\infty)$ is continuous and increasing,
    $r_j(0)=0$ and $\lim_{t \to + \infty} r_j(t) = + \infty$, and we define the maximal operator
    $$\W_{r,*}^\lambda(f)
        = \sup_{t>0} |\W_{r(t)}^\lambda(f)|,$$
    where
    $$\W_{r(t)}^\lambda(f)(x)
        = \int_{(0,\infty)^n} \W_{r(t)}^\lambda(x,y) f(y) d\mu_\lambda(y),
        \quad f \in L^p((0,\infty)^n, d\mu_\lambda), \ 1 \leq p \leq \infty,$$
    and
    $$\W_{r(t)}^\lambda(x,y)
        = \prod_{j=1}^n W_{r_j(t)}^{\lambda_j}(x_j,y_j),
        \quad x,y \in (0,\infty)^n \text{ and } t>0.$$
    It is clear that if $r_j(t)=t$, $t \geq 0$, $j=1, \dots, n$, then $\W_{r,*}^\lambda=\W_*^\lambda$.\\

    Our first result is the following one.

    \begin{Th}\label{Th1.1}
        Suppose that $\lambda \in (-1/2,\infty)^n$ and $r$ is a function as above. Then, the maximal operator
        $\W_{r,*}^\lambda$ is bounded from $L^p((0,\infty)^n, d\mu_\lambda)$ into itself, for every
        $1 < p \leq \infty$, and from $L^1((0,\infty)^n, d\mu_\lambda)$ into
        $L^{1,\infty}((0,\infty)^n, d\mu_\lambda)$.
    \end{Th}

    An immediate consequence of Theorem~\ref{Th1.1} is the next convergence result.

    \begin{Cor}\label{Cor1.2}
        Let $\lambda \in (-1/2,\infty)^n$ and $r$ be a function as above. Then,
        for every $f \in L^p((0,\infty)^n, d\mu_\lambda)$, $1 \leq p < \infty$,
        $$\lim_{t \to 0^+} \W_{r(t)}^\lambda(f)(x)=f(x), \quad \text{a.e. } x \in (0,\infty)^n.$$
    \end{Cor}

    We now consider the Laguerre operator $\mathcal{L}_\lambda$, $\lambda>-1/2$, defined by
    $$\mathcal{L}_\lambda = \Delta_\lambda + \frac{x^2}{4}, \quad \text{on } (0,\infty).$$
    Also, for every $k \in \mathbb{N}$, we define the $k$-th Laguerre function $\psi_k^\lambda$ by
    $$\psi_k^\lambda(x)
        = 2^{-(2\lambda-1)/4} \left( \frac{k!}{\G(k + \lambda + 1/2)} \right)^{1/2}
            L_k^{\lambda-1/2}\left( \frac{x^2}{2}\right)e^{-x^2/4},
        \quad x \in (0,\infty),$$
    where $L_k^\alpha$ denotes the $k$-th Laguerre polynomial with parameter $\alpha>-1$.
    The system $\{\psi_k^\lambda\}_{k \in \mathbb{N}}$ is a complete orthonormal family in
    $L^2((0,\infty), x^{2\lambda}dx)$. Moreover,
    $$\mathcal{L}_\lambda (\psi_k^\lambda) = (2k+\lambda +1/2) \psi_k^\lambda, \quad k \in \mathbb{N}.$$
    We extend the definition of the operator $\mathcal{L}_\lambda$ as follows
    $$\mathcal{L}_\lambda (f)
        = \sum_{k=0}^\infty (2k+\lambda +1/2) \langle f,\psi_k^\lambda\rangle \psi_k^\lambda,
        \quad f \in D(\mathcal{L}_\lambda), $$
    where $\langle \cdot , \cdot \rangle$ denotes the usual inner product in $L^2((0,\infty), x^{2\lambda}dx)$,
    and
    $$D(\mathcal{L}_\lambda)
        = \{f \in L^2((0,\infty), x^{2\lambda}dx) :
        \sum_{k=0}^\infty (2k+\lambda+1/2)^2 |\langle f,\psi_k^\lambda\rangle|^2<\infty \}.$$
    Thus, $\mathcal{L}_\lambda$ is positive and selfadjoint in $L^2((0,\infty), x^{2\lambda}dx)$.
    Moreover, $-\mathcal{L}_\lambda$ generates a diffusion semigroup $\{L_t^\lambda\}_{t>0}$ on
    $L^2((0,\infty), x^{2\lambda}dx)$ where, for every $t>0$,
    \begin{equation}\label{1.2}
        L_t^\lambda(f)(x)
            = \int_0^\infty L_t^\lambda(x,y) f(y) y^{2\lambda} dy,
            \quad  f \in L^2((0,\infty), x^{2\lambda}dx), \  x,t \in  (0,\infty),
    \end{equation}
    being
    $$L_t^\lambda(x,y)
        = \frac{e^{-t}}{1-e^{-2t}} (xy)^{-\lambda+1/2}
            I_{\lambda-1/2}\left( \frac{e^{-t}xy}{1-e^{-2t}} \right)
            \exp\left( -\frac{1}{4} \frac{1+e^{-2t}}{1-e^{-2t}} (x^2 + y^2)\right),
        \quad x,y,t \in (0,\infty).$$
    Moreover, \eqref{1.2} defines also a diffusionp semigroup in $L^p((0,\infty), x^{2\lambda}dx)$,
    $1 \leq p \leq \infty$.\\

    Suppose now that $\lambda \in (-1/2,\infty)^n$. The $n$-dimensional heat Laguerre semigroup
    $\{\LLL_t^\lambda\}_{t>0}$ is defined as follows. For every $t>0$, $f \in L^p((0,\infty)^n, d\mu_\lambda)$,
    $1 \leq p \leq \infty$, we write
    $$\LLL_t^\lambda (f)(x)
        = \int_{(0,\infty)^n} \LLL_t^\lambda(x,y) f(y) d\mu_\lambda(y),
        \quad x \in (0,\infty)^n,$$
    being
    $$\LLL_t^\lambda(x,y)
        = \prod_{j=1}^n L_t^{\lambda_j}(x_j,y_j),
        \quad x,y \in (0,\infty)^n, \ t>0.$$
    In \cite[Theorem 1.1]{NS1} it was showed that the maximal operator $\LLL_*^\lambda$,
    defined by
    $$\LLL_*^\lambda(f)
        = \sup_{t>0} |\LLL_t^\lambda(f)|,$$
    is bounded from $L^1((0,\infty)^n, d\mu_\lambda)$ into $L^{1,\infty}((0,\infty)^n, d\mu_\lambda)$
    by employing an ingenious but long and not easy procedure.\\

    Assume that the function $r : [0,\infty) \longrightarrow [0,\infty)^n$ is as in Theorem~\ref{Th1.1}.
    We define the maximal operator $\LLL_{r,*}^\lambda$ by
    $$\LLL_{r,*}^\lambda(f)
        = \sup_{t>0} |\LLL_{r(t)}^\lambda(f)|,$$
    where
    $$\LLL_{r(t)}^\lambda (f)(x)
        = \int_{(0,\infty)^n} \LLL_{r(t)}^\lambda(x,y) f(y) d\mu_\lambda(y),
        \quad x \in (0,\infty)^n, \ t>0,$$
    being
    $$\LLL_{r(t)}^\lambda(x,y)
        = \prod_{j=1}^n L_{r_j(t)}^{\lambda_j}(x_j,y_j),
        \quad x,y \in (0,\infty)^n, \ t>0.$$

    Since $|L_t^\lambda(f)| \leq W_t^\lambda(|f|)$, $t>0$, from Theorem~\ref{Th1.1}
    we deduce the following result that includes as a special case \cite[Theorem 1.1]{NS1}.

    \begin{Th}\label{Th1.2}
        Suppose that $\lambda \in (-1/2,\infty)^n$ and $r$ is as in Theorem~\ref{Th1.1}.
        Then, the maximal operator
        $\LLL_{r,*}^\lambda$ is bounded from $L^p((0,\infty)^n, d\mu_\lambda)$ into itself, for every
        $1 < p \leq \infty$, and from $L^1((0,\infty)^n, d\mu_\lambda)$ into
        $L^{1,\infty}((0,\infty)^n, d\mu_\lambda)$.
    \end{Th}

    If we denote, for every $k=(k_1, \dots, k_n) \in \mathbb{N}^n$, and $\lambda \in (-1/2,\infty)^n$,
    $\psi_k^\lambda(x)= \overset{n}{\underset{j=1}{\prod}} \psi_{k_j}^{\lambda_j}(x_j)$,
    $x \in (0,\infty)^n$, the subspace
    $\spann \{\psi_k^\lambda\}_{k \in \mathbb{N}^n}$ is dense in $L^p((0,\infty)^n, d\mu_\lambda)$,
    $1 \leq p < \infty$. For every $f \in \spann \{\psi_k^\lambda\}_{k \in \mathbb{N}^n}$, we have that
    $$\LLL_{r(t)}^\lambda (f)
        = \sum_{k \in \mathbb{N}^n} e^{-\overset{n}{\underset{j=1}{\sum}}r_j(t)(2k_j+\lambda_j+1/2)}
            \langle f , \psi_k^\lambda \rangle \psi_k^\lambda.$$
    Since this last sum has at most a finite number of terms it is clear that
    $\lim_{t \to 0^+} \LLL_{r(t)}^\lambda(f)(x) = f(x), \quad x \in (0,\infty)^n,$ for every
    $f \in \spann \{\psi_k^\lambda\}_{k \in \mathbb{N}^n}$. Then, standard arguments allow us to deduce the following
    convergence result.

    \begin{Cor}\label{Cor1.4}
        Let $\lambda \in (-1/2,\infty)^n$ and $r$ be as in Theorem~\ref{Th1.1}. Then,
        for every $f \in L^p((0,\infty)^n, d\mu_\lambda)$, $1 \leq p < \infty$,
        $$\lim_{t \to 0^+} \LLL_{r(t)}^\lambda(f)(x)=f(x), \quad \text{a.e. } x \in (0,\infty)^n.$$
    \end{Cor}

    In the next section we present the proofs of Theorems~\ref{Th1.1} and Corollary~\ref{Cor1.2}.\\

    Throughout the paper by $C$ and $c$ we denote positive constants that can change from one line to the other.

    \section{Proof of the results}  \label{sec:proofs}

    In order to prove Theorem~\ref{Th1.1} we need some properties of the Bessel heat kernel $W_r^\lambda(x,y)$, $r,x,y \in (0,\infty)$,
    $\lambda>-1/2$.\\

    By proceeding as in the proof of \cite[Lemma 3.1]{BHNV} we can show the following result.

    \begin{Lem}\label{Lem2.1}
        Let $\lambda>-1/2$. Then, for every $r,x,y  \in (0,\infty)$,
        \begin{numcases}{W_r^\lambda(x,y) \leq C}
            x^{-2\lambda-1} e^{-cx^2/r}, & $0<y \leq x/2$; \label{2.1}\\
            x^{-2\lambda-1} e^{-cx^2/r} + \dfrac{(xy)^{-\lambda}}{\sqrt{r}}e^{-(x-y)^2/4r}, & $x/2<y <2x$; \label{2.2} \\
            y^{-2\lambda-1} e^{-cy^2/r}, & $0<2x \leq y$. \label{2.3}
        \end{numcases}
    \end{Lem}

    According to \cite[Chapter VI, Section 6.15]{Wat}, if $\nu>-1/2$ we can write
    $$I_\nu(z)
        = \frac{z^\nu}{\sqrt{\pi} 2^\nu \G(\nu+1/2)} \int_{-1}^1 e^{-zs}(1-s^2)^{\nu-1/2}ds,
        \quad z \in (0,\infty).$$
    Moreover, $I_\nu(z) = 2(\nu+1) I_{\nu+1}(z)/z + I_{\nu+2}(z)$, $ z \in (0,\infty)$ and $\nu>-1$ (\cite[Chapter III, Section $3 \cdot 71$]{Wat}).
    Hence, if $\lambda>-1/2$ we obtain, for every $z \in (0,\infty)$,
    \begin{align*}
        I_{\lambda-1/2}(z)
            = & \frac{2\lambda + 1}{z}I_{\lambda+1/2}(z) + I_{\lambda+3/2}(z)\\
            = & \frac{(2\lambda + 1)z^{\lambda-1/2}}{\sqrt{\pi} 2^{\lambda+1/2} \G(\lambda+1)} \int_{-1}^1 e^{-zs}(1-s^2)^{\lambda}ds
              + \frac{z^{\lambda+3/2}}{\sqrt{\pi} 2^{\lambda+3/2} \G(\lambda+2)} \int_{-1}^1 e^{-zs}(1-s^2)^{\lambda+1}ds.
    \end{align*}
    Then, the Bessel heat kernel can be written as
    \begin{align}\label{2.4}
        W_r^\lambda(x,y)
            = & \frac{1}{\sqrt{\pi} 2^{2\lambda+1} \G(\lambda+1)} \Big( \frac{2\lambda+1}{r^{\lambda+1/2}} \int_{-1}^1 e^{-(x^2+y^2+2xys)/4r}(1-s^2)^{\lambda}ds \nonumber \\
            & + \frac{(xy)^2}{2^3(\lambda+1)r^{\lambda+5/2}} \int_{-1}^1 e^{-(x^2+y^2+2xys)/4r}(1-s^2)^{\lambda+1}ds \Big),
            \quad r,x,y \in (0,\infty),
    \end{align}
    where $\lambda>-1/2$.\\

    The key result to show Theorem~\ref{Th1.1} is the following.

    \begin{Prop}\label{Prop2.1}
        Let $\lambda>-1/2$. Then, there exist $C,c>0$ such that
        $$W_r^\lambda(x,y)
            \leq C \sum_{k=0}^\infty \frac{e^{-c2^{2k}}}{\mu_\lambda(I_k(x,r))} \chi_{I_k(x,r)}(y),
            \quad r,x,y \in (0,\infty),$$
        where $I_k(x,r)=[x-2^k\sqrt{r},x+2^k\sqrt{r}] \cap (0,\infty)$, $r,x \in (0,\infty)$ and $k \in \mathbb{N}$.
    \end{Prop}

    \begin{proof}
        Let $r,x \in (0,\infty)$. We consider different cases.\\

        Suppose that $x \leq \sqrt{r}$. Then, $I_0(x,r)=[0,x+\sqrt{r}]$ and
        $$\mu_\lambda(I_0(x,r))
            =\dfrac{(x+\sqrt{r})^{2\lambda+1}}{2\lambda+1}
            \leq C r^{\lambda+1/2}.$$
        Since $x^2+y^2+2xys=(x-y)^2+2xy(1+s) \geq 0$, $y \in (0,\infty)$ and $s \in (-1,1)$, from \eqref{2.4}
        we deduce that
        \begin{align}\label{2.5}
            W_r^\lambda(x,y)
                \leq & \frac{C}{r^{\lambda+1/2}} \left( 1 + \left(\frac{xy}{r}\right)^2  \right)
                \leq \frac{C}{r^{\lambda+1/2}} \left( 1 + \left(\frac{x(x+\sqrt{r})}{r}\right)^2  \right)
                \leq \frac{C}{r^{\lambda+1/2}} \nonumber \\
                \leq & \frac{C}{\mu_\lambda(I_0(x,r))},
                \quad y \in I_0(x,r).
        \end{align}

        Assume now that $x>\sqrt{r}$. Then, $I_0(x,r)=[x-\sqrt{r},x+\sqrt{r}]$ and
        $$\mu_\lambda(I_0(x,r))
            =\dfrac{1}{2\lambda+1} \left( (x+\sqrt{r})^{2\lambda+1} - (x-\sqrt{r})^{2\lambda+1} \right).$$
        The mean value theorem leads to $\mu_\lambda(I_0(x,r)) = 2 \sqrt{r}u^{2\lambda}$, for a certain
        $u \in (x-\sqrt{r},x+\sqrt{r})$. If $\lambda \geq 0$, it follows that
        $\mu_\lambda(I_0(x,r)) \leq 2 \sqrt{r}(x+\sqrt{r})^{2\lambda}$. On the other hand, if $-1/2<\lambda<0$, we distinguish
        two cases.
        \begin{itemize}
            \item If $x \in (\sqrt{r},3\sqrt{r})$, then
            $$\mu_\lambda(I_0(x,r))
                \leq \int_0^{x+\sqrt{r}} y^{2\lambda} dy
                \leq C (x+\sqrt{r})^{2\lambda+1}
                \leq C \sqrt{r}(x+\sqrt{r})^{2\lambda}.$$
            \item If $x \geq 3\sqrt{r}$, then
            $$\mu_\lambda(I_0(x,r))
                \leq C \sqrt{r}(x-\sqrt{r})^{2\lambda}
                \leq C \sqrt{r}\left(\frac{x+\sqrt{r}}{2}\right)^{2\lambda}.$$
        \end{itemize}
        Hence, we conclude that
        $\mu_\lambda(I_0(x,r))
            \leq C \sqrt{r}x^{2\lambda}
            \leq C x^{2\lambda+1}$
        in either case. By taking in mind Lemma~\ref{Lem2.1} in order to estimate $W_r^\lambda(x,y)$ we distinguish three regions.
        Firstly, by \eqref{2.1} it follows that
        $$W_r^\lambda(x,y)
            \leq C x^{-2\lambda-1}
            \leq \frac{C}{\mu_\lambda(I_0(x,r))},
            \quad 0 < y \leq x/2,$$
        and from \eqref{2.3} we deduce that
        $$W_r^\lambda(x,y)
            \leq C y^{-2\lambda-1}
            \leq C x^{-2\lambda-1}
            \leq \frac{C}{\mu_\lambda(I_0(x,r))},
            \quad 2x \leq y .$$
        Moreover, \eqref{2.2} implies that
        $$W_r^\lambda(x,y)
            \leq C \left( x^{-2 \lambda-1} + \frac{x^{-2 \lambda}}{\sqrt{r}} \right)
            \leq \frac{C}{\mu_\lambda(I_0(x,r))},
            \quad x/2 < y < 2x.$$
        We obtain that
        \begin{equation}\label{2.6}
            W_r^\lambda(x,y)
                \leq \frac{C}{\mu_\lambda(I_0(x,r))},
                \quad y \in (0,\infty).
        \end{equation}

        Suppose now that $k \in \mathbb{N}\setminus\{0\}$. We define
        $C_k(x,r)=\{y \in (0,\infty) : 2^{k-1} \sqrt{r} < |x-y| \leq 2^k \sqrt{r}\}$.
        It is clear that $C_k(x,r) \subset I_k(x,r)$.\\

        Assume that $x \leq 2^k \sqrt{r}$. Then, $I_k(x,r)=[0,x+2^k\sqrt{r}]$ and
        $\mu_\lambda(I_k(x,r)) \leq C (2^k\sqrt{r})^{2\lambda+1}.$
        According to \eqref{2.4}, since $x^2+y^2+2xys=(x-y)^2+2xy(1+s)$, $y \in (0,\infty)$ and $s \in (-1,1)$,
        we have that
        \begin{align}\label{2.7}
            W_r^\lambda(x,y)
                \leq & C \frac{e^{-c2^{2k}}}{r^{\lambda+1/2}} \left( 1 + \left(\frac{x(x+2^k\sqrt{r})}{r} \right)^2 \right)
                \leq   C \frac{2^{4k}e^{-c2^{2k}}}{r^{\lambda+1/2}}
                \leq  C\frac{e^{-c2^{2k}}}{\mu_\lambda(I_k(x,r))},
                \quad y \in C_k(x,r).
        \end{align}

        We take now $x>2^k \sqrt{r}$. Then, $I_k(x,r)=[x-2^k\sqrt{r},x+2^k\sqrt{r}]$ and by proceeding
        as above we get
        $\mu_\lambda(I_k(x,r))  \leq C 2^{k}\sqrt{r} x^{2\lambda}  \leq C x^{2\lambda+1}.$
        We distinguish again three cases. If $0<y \leq x/2$ and $y \in C_k(x,r)$
        we have that $2^{k-1}\sqrt{r} \leq x \leq 2^{k+1}\sqrt{r}$.
        Then, \eqref{2.1} implies that
        $$W_r^\lambda(x,y)
            \leq  C\frac{e^{-c2^{2k}}}{\mu_\lambda(I_k(x,r))},
            \quad 0<y \leq x/2.$$
        Also, from \eqref{2.3} we deduce
        $$W_r^\lambda(x,y)
            \leq  C\frac{e^{-c2^{2k}}}{\mu_\lambda(I_k(x,r))},
            \quad 2x \leq y.$$
        Finally, by \eqref{2.2} if follows that
        $$W_r^\lambda(x,y)
            \leq  Ce^{-c2^{2k}} \left( x^{-2\lambda-1} + \frac{x^{-2\lambda}}{\sqrt{r}} \right)
            \leq  C\frac{e^{-c2^{2k}}}{\mu_\lambda(I_k(x,r))},
            \quad x/2<y<2x \text{ and } y \in C_k(x,r).$$
        Hence, we get
        \begin{equation}\label{2.8}
            W_r^\lambda(x,y)
                \leq  C\frac{e^{-c2^{2k}}}{\mu_\lambda(I_k(x,r))},
                \quad y \in C_k(x,r).
        \end{equation}
        By combining \eqref{2.5}, \eqref{2.6}, \eqref{2.7} and \eqref{2.8} we obtain
        \begin{align*}
            W_r^\lambda(x,y)
                = & W_r^\lambda(x,y) \chi_{I_0(x,r)}(y)  + \sum_{k=1}^\infty W_r^\lambda(x,y) \chi_{C_k(x,r)}(y)
                \leq  C \left( \frac{\chi_{I_0(x,r)}(y)}{\mu_\lambda(I_0(x,r))}
                    +  \sum_{k=1}^\infty \frac{e^{-c2^{2k}}\chi_{C_k(x,r)}(y)}{\mu_\lambda(I_k(x,r))}   \right) \\
                \leq & C \sum_{k=0}^\infty \frac{e^{-c2^{2k}}}{\mu_\lambda(I_k(x,r))} \chi_{I_k(x,r)}(y),
                \quad y \in (0,\infty).
        \end{align*}
    \end{proof}

    \subsection{}\label{subsec:Th1.1} \emph{Proof of Theorem~\ref{Th1.1}}\\

    According to Proposition~\ref{Prop2.1} we have that
    \begin{align*}
        |\W_{r(t)}^\lambda(f)(x)|
            \leq & \int_{(0,\infty)^n} \prod_{j=1}^n W_{r_j(t)}^{\lambda_j}(x_j,y_j) |f(y)| d\mu_\lambda(y) \\
            \leq & K \sum_{k \in \mathbb{N}^n} \prod_{j=1}^n e^{-c2^{2k_j}}
                    \frac{1}{\mu_\lambda(R_k(x,r(t)))} \int_{R_k(x,r(t))} |f(y)| d\mu_\lambda(y),
            \quad x \in (0,\infty)^n \text{ and } t>0,
    \end{align*}
    where $R_k(x,r(t))=\overset{n}{\underset{j=1}{\prod}} I_{k_j}(x_j,r_j(t))$ and $K>0$.\\

    Then, it follows that
    \begin{align}\label{2.9}
        |\W_{r,*}^\lambda(f)(x)|
            \leq & K \sum_{k \in \mathbb{N}^n} \left(\prod_{j=1}^n e^{-c2^{2k_j}} \right) \mathcal{M}_{r,k}^\lambda(f)(x),
            \quad x \in (0,\infty)^n,
    \end{align}
    where $\mathcal{M}_{r,k}^\lambda$ represents the maximal function defined by
    $$\mathcal{M}_{r,k}^\lambda(f)(x)
        = \sup_{t>0} \frac{1}{\mu_\lambda(R_k(x,r(t)))} \int_{R_k(x,r(t))} |f(y)| d\mu_\lambda(y),
        \quad x \in (0,\infty)^n.$$

    By \cite[Theorem 1]{CC1}, for every $k \in \mathbb{N}^n$ and $\gamma>0$, we get
    \begin{equation}\label{2.10}
        \mu_\lambda\left( \{x\in (0,\infty)^n : \mathcal{M}_{r,k}^\lambda(f)(x) > \gamma \} \right)
            \leq \frac{6^n n!}{\gamma} \|f\|_{L^1((0,\infty)^n,d\mu_\lambda)},
            \quad f \in L^1((0,\infty)^n,d\mu_\lambda).
    \end{equation}

    Since
    $$\sum_{k \in \mathbb{N}^n} \prod_{j=1}^n e^{-\alpha 2^{2k_j}}
        = \left( \sum_{m=0}^\infty e^{-\alpha 2^{2m}} \right)^n< \infty,
        \quad \text{when } \alpha>0,$$
    by defining
    $$Q_k
        = \left( 2 K \sum_{\ell \in \mathbb{N}^n} \prod_{j=1}^n e^{-c 2^{2\ell_j-1}}  \right)^{-1} \prod_{j=1}^n e^{c  2^{2k_j-1}},
        \quad k \in \mathbb{N}^n,$$
    we have that
    \begin{align*}
        \left\{ x\in (0,\infty)^n : |\W_{r,*}^\lambda(f)(x)| > \gamma  \right\}
            \subset \bigcup_{k \in \mathbb{N}^n}  \left\{ x\in (0,\infty)^n : \mathcal{M}_{r,k}^\lambda(f)(x) > \gamma Q_k  \right\}.
    \end{align*}
    Hence, from \eqref{2.10} we deduce
    \begin{align*}
        & \mu_\lambda\left( \{x\in (0,\infty)^n : |\W_{r,*}^\lambda(f)(x)| > \gamma \} \right)
            \leq  \sum_{k \in \mathbb{N}^n} \mu_\lambda\left( \left\{ x\in (0,\infty)^n : \mathcal{M}_{r,k}^\lambda(f)(x) > \gamma Q_k  \right\}  \right) \\
        & \qquad \qquad    \leq  2 K \frac{6^n n!}{\gamma}
                    \left(\sum_{k \in \mathbb{N}^n} \prod_{j=1}^n e^{-c 2^{2k_j-1}}\right)
                    \left( \sum_{\ell \in \mathbb{N}^n} \prod_{j=1}^n e^{-c 2^{2\ell_j-1}} \right)
                    \|f\|_{L^1((0,\infty)^n,d\mu_\lambda)},
            \quad \gamma >0.
    \end{align*}
    Thus we prove that $\W_{r,*}^\lambda$ is bounded from $L^1((0,\infty)^n,d\mu_\lambda)$ into $L^{1,\infty}((0,\infty)^n,d\mu_\lambda)$.\\

    According to \eqref{2.9} it is clear that $\W_{r,*}^\lambda$ is bounded from $L^\infty((0,\infty)^n,d\mu_\lambda)$ into $L^\infty((0,\infty)^n,d\mu_\lambda)$.
    Then, by interpolating we conclude that $\W_{r,*}^\lambda$ is bounded from $L^p((0,\infty)^n,d\mu_\lambda)$ into itself, for every $1<p<\infty$.
    \begin{flushright}
        \qed
    \end{flushright}

    \subsection{}\label{subsec:Cor1.2}\emph{Proof of Corollary~\ref{Cor1.2}}\\

    In order to show this theorem it is sufficient to see that, for every $f \in C_c^\infty((0,\infty)^n)$, the space of
    smooth functions with compact support on $(0,\infty)^n$,
    $$\lim_{t \to 0^+} \W_{r(t)}^\lambda(f)(x)
        = f(x),
        \quad x \in (0,\infty)^n.$$
    Let $f \in C_c^\infty((0,\infty)^n)$. The Hankel transform $h_\lambda(f)$ of $f$ is defined by
    $$h_\lambda(f)(x)
        = \int_{(0,\infty)^n} \prod_{j=1}^n (x_jy_j)^{-\lambda_j+1/2} J_{\lambda_j-1/2}(x_jy_j) f(y)d\mu_\lambda(y),
        \quad x \in (0,\infty)^n.$$
    According to \eqref{0.1} we deduce that
    $$\W_{r(t)}^\lambda(f)(x)
        = h_\lambda \left( \prod_{j=1}^n e^{-y_j^2r_j(t)} h_\lambda(f)(y) \right)(x),
        \quad x \in (0,\infty)^n.$$
    By using the dominated convergence theorem we conclude that
    $$\lim_{t \to 0^+} \W_{r(t)}^\lambda(f)(x)
        = h_\lambda(h_\lambda(f))(x),
        \quad x \in (0,\infty)^n,$$
    and the proof finishes because $h_\lambda^{-1}=h_\lambda$ in $L^2((0,\infty)^n,d\mu_\lambda)$
    (see \cite[p. 125]{BS}).
    \begin{flushright}
        \qed
    \end{flushright}


\end{document}